\documentclass[12pt, a4paper]{article}

\usepackage{amsthm}
\usepackage{cite}
\usepackage{mathtools}
\usepackage{enumitem}
\usepackage{hyperref}
\usepackage{graphicx}
\usepackage{amssymb}
\usepackage{xcolor}
\usepackage{multirow}
\usepackage[all]{xy}
\usepackage{tikz}
\usepackage{booktabs}
\usepackage{arydshln}
\usepackage{mleftright}
\usepackage{pgffor}
\usepackage{blkarray}
\usepackage{subfigure}
\usepackage{authblk}
\usepackage[utf8]{inputenc}
\usepackage[T1]{fontenc}
\usepackage{amsmath}
\usepackage{amsfonts}
\usepackage{caption}
\usepackage{float}
\usepackage{ifthen}
\usepackage{tikz-cd}
\usepackage{longtable}
\usepackage{fullpage}
\usepackage{textcomp}
\usepackage{makecell}
\usepackage{setspace}
\usepackage{color}
\usepackage{chngcntr}
\usepackage[misc]{ifsym}

\hypersetup{
  colorlinks,
  linkcolor={blue!50!black},
  citecolor={blue!50!black},
  urlcolor={blue!80!black}
}

\usetikzlibrary{patterns}
\usetikzlibrary{intersections}
\usetikzlibrary{calc}

\setlist[enumerate]{labelsep=*, leftmargin=1.5pc}
\setlist[enumerate]{label=\normalfont(\roman*), ref=\roman*}
\usepackage[margin=2.7cm]{geometry}
\graphicspath{{images/}}
\theoremstyle{plain}
\newtheorem{thm}{Theorem}[section]
\newtheorem{pro}[thm]{Proposition}
\newtheorem{lem}[thm]{Lemma}
\newtheorem{rmk}[thm]{Remark}
\newtheorem{cor}[thm]{Corollary}

\theoremstyle{definition}
\newtheorem{dfn}[thm]{Definition}

\newtheorem{eg}[thm]{Example}

\newtheorem{con}[thm]{Construction}

\DeclareMathOperator{\codim}{codim}

\DeclareMathOperator{\Spec}{Spec}

\DeclareMathOperator {\rk} {rank}

\newcommand {\logdim} [1][]{\operatorname{logdim}_{\ifthenelse{\equal{#1}{}}{k}{#1}}}
\newcommand {\ul} [1]{\underline{#1}}
\newcommand {\Log} [1]{\mathcal{L}_{#1}}
\newcommand{\catname}[1]{{\normalfont\textbf{#1}}}

\def \Post {\text{\Letter}}

\def \Alphabet {A,B,C,D,E,F,G,H,I,J,K,L,M,N,O,P,Q,R,S,T,U,V,W,X,Y,Z}
\def \alphabet {a,b,c,d,e,f,g,h,i,j,k,l,m,n,o,p,q,r,s,t,u,v,w,x,y,z}

\newcommand{\one}{\mathbf{1}}

\foreach \x in \Alphabet {\expandafter\xdef\csname\x\x\endcsname{\noexpand\mathbb{\x}}}
\foreach \x in \Alphabet {\expandafter\xdef\csname b\x\endcsname{\noexpand\mathbf{\x}}}
\foreach \x in \Alphabet {\expandafter\xdef\csname c\x\endcsname{\noexpand\mathcal{\x}}}
\foreach \x in \alphabet {\expandafter\xdef\csname cal\x\endcsname{\noexpand\mathcal{\x}}}
\foreach \x in \Alphabet {\expandafter\xdef\csname fr\x\endcsname{\noexpand\mathfrak{\x}}}

\newcommand \CharSheaf [1] {\overline{\mathcal{M}}_{#1}}
\newcommand{\ComSquare}[8]{
\begin{tikzcd}[ampersand replacement=\&]
  #1 \arrow{r}{#5} \arrow{d}{#8} \& #2 \arrow{d}{#6} \\
#4 \arrow{r}{#7} \& #3
\end{tikzcd}
}

\begin{document}

\title{
	{Bivariant classes and Logarithmic Chow theory}\\
}
\author{Lawrence Jack Barrott}

\maketitle

\begin{abstract}

We describe a refined Chow theory for log schemes extending the theory of $b$-Chow suggested in section 9 of~\cite{MultiplicativityOfTheDoubleRamificationCycle} based off of a definition appearing in~\cite{Shokurov}. This produces a dimension graded family of Abelian groups supporting a push-forward and pull-back along proper and log flat morphisms respectively, together with a bivariant theory satisfying Poincar\'e Duality. In the final section we relate this construction to the construction of the log normal cone in~\cite{LeosPaper}.

\end{abstract}

\section{Introduction}

The \emph{Chow groups} of a scheme $X$ classify integral subschemes of $X$ up to rational equivalent. Two integral subschemes $V_0$ and $V_1$ are \emph{rationally equivalent} if there is a flat family over $\PP^1$ such that $V_0$ and $V_1$ appear as fibres of this family. The Chow group $A_k (X)$ is then the quotient of the free Abelian group generated by integral $k$-dimensional subschemes of $X$, $B_k (X)$, modulo the subgroup generated by the above relation. We call elements of $A_k (X)$ \emph{cycle classes} and elements of $B_k (X)$ \emph{cycles}. These satisfy a collection of properties whose proofs can be found in~\cite{Fulton}:

\begin{enumerate}
\item If $f: X \to Y$ is a proper map then there is a well-defined pushforward morphism $f_*: A_k(X) \to A_k (Y)$.
\item If $f: X \to Y$ is a flat map of relative dimension $d$ then there is a well-defined pullback morphism $f^*: A_k(Y) \to A_{k+d} (X)$.
\item If $f: X \to Y$ is a morphism of relative dimension $d d = \dim X - \dim Y$ with $X$ and $Y$ both smooth over $\Spec k$ then there is a well-defined Gysin pullback morphism $f^! : A_k(Y) \to A_{k+d} (X)$.
\item If $f: X \to Y$ is a vector bundle of rank $d$ then $f^*$ is an isomorphism.
\item If $i: D \to X$ is the inclusion of a closed subscheme and $j:U \to X$ the open complement then the following is an exact sequence, called the \emph{excision sequence}:
\[ A_k (D) \xrightarrow{i_*} A_k (X) \xrightarrow{j^*} A_k (U) \to 0 \]
\end{enumerate}

For a smooth scheme $X$ the Gysin pullback along the diagonal map $X \to X \times X$ induces a ring structure on $A_* (X)$. This structure is compatible with Gysin maps and flat pullback, but not pushforward.

Given any scheme $X$ one can define the \emph{large $b$-Chow group}, this is the limit over the Chow group of all proper birational morphisms, or modifications, $\tilde{X} \to X$, with transition maps pushforward along all birational morphisms as schemes over $X$. This is the definition used by Aluffi in~\cite{Aluffi}, see Definition 2.5. It does not however satisfy any of the properties listed above. The property of being birational is not stable under formation of fibre products.

Instead one could consider the small $b$-Chow ring for a scheme $X$ for which resolution of singularities holds. This is the \emph{colimit} over the Chow groups of all proper birational morphisms $\tilde{X} \to X$ with $\tilde{X}$ smooth and transition maps Gysin pullbacks along all proper birational morphisms as schemes over $X$. As the name suggests this supports an intersection product and acts on the above large $b$-Chow group, however for the same reason as above it does not satisfy any other properties. It also is not well defined unless resolution of singularities holds, so for non-reduced schemes this presents issues.

More generally one could consider the colimits of the Chow cohomology rings $A^* (X)$, described by Fulton in~\cite{Fulton}, and used by Aluffi in~\cite{Aluffi} in the case of divisors, see Definition 2.8. When $X$ satisfies resolution of singularities these map to the inverse limit of the Chow groups under pushforward, as the small $b$-Chow groups embed into the large $b$-Chow groups. However these still do not satisfy the above properties and are rarely calculable outside of the smooth case, although see~\cite{BlochGilletSoule} Section 1.6 for one example where descriptions are known and~\cite{vG2},~\cite{vG} for an application of this description.

\subsection{Log geometry}

In this paper we propose a solution to some of these problems, extending an idea in section 9 of~\cite{MultiplicativityOfTheDoubleRamificationCycle} and applying it to the construction of~\cite{LeosPaper}. The spaces these authors study have a naturally defined \emph{log structure}.

The theory of \emph{log geometry} was introduced by Kato, based on ideas of Deligne and Illusie, in~\cite{Kato}. It studies schemes $\ul{X}$ equipped with a log structure consisting of a sheaf of monoids $\cM_X$, together with a morphism $\alpha_X : \cM_X \to \cO_{\ul{X}}$ which we require to be an isomorphism on the submonoid $\cO_{\ul{X}}^*$. Associated to this is another sheaf, the \emph{characteristic sheaf}, the quotient $\cM_X / \alpha^{-1} \cO_X^*$. There are notions of log flatness, log \'etaleness and log smoothness satisfying similar properties to the classical notion. We focus on the category of \emph{fs} log schemes, whose definition can be found in~\cite{Kato}. In the introduction we will include the phrase fs to emphasise the difference from the classical construction, in the body however we will leave it implicit. We will conflate $\Spec k$ with \emph{the trivial log point}, the log scheme with $\ul{X} = \Spec k$, $\cM_X \cong k^*$ and $\alpha_X$ the inclusion.

In nice cases, for instance on the moduli space of stable curves, this structure specifies a collection of normal crossings divisors. Then section 9 of~\cite{MultiplicativityOfTheDoubleRamificationCycle}, and earlier work of Holmes in~\cite{ExtendingTheDoubleRamificationCycleByResolvingTheAbelJacobiMap}, suggest looking at only \emph{log blow-ups}, weighted blow ups of intersections of these divisors. In nice situations (when $X$ is \emph{log smooth} over the trivial log point) this presents a much smaller category of blow-ups but still allows the singularities to be resolved and the small $b$-Chow construction to be applied.

This construction however still only applies for log smooth schemes. In this paper we extend the above construction to all log schemes by constructing Gysin maps even when the underlying schemes are not smooth but so long as the log structure is \emph{locally  free}, locally isomorphic to $\cO_{\ul{X}}^* \oplus \NN^r$. This produces the following theorem, a translation of the above list, whose proofs are unfortunately necessarily scattered throughout this paper:

\begin{thm}[Constructions~\ref{con:pushforward} and~\ref{con:pullback}, Example~\ref{eg:SmoothGysin} and Corollary~\ref{cor:excision} respectively]

  There is a functor $A_* ^\dagger: \catname{(fs) Log Sch} \to \catname {Gr Ab}$ from the category of (fs) log schemes of finite type over $\Spec k$ to the category of graded Abelian groups.

  \begin{enumerate}
\item If $f: X \to Y$ is a proper map then there is a well-defined pushforward morphism $f_*: A^\dagger_k(X) \to A^\dagger_k (Y)$.
\item If $f: X \to Y$ is a log flat map of relative dimension $d$ then there is a well-defined pullback morphism $f^*: A^\dagger _k(Y) \to A^\dagger_{k+d} (X)$.
    \item If $f: X \to Y$ is a morphism of relative dimension $d = \dim X - \dim Y$ with $X$ and $Y$ both log smooth over $\Spec k$ then there is a well-defined Gysin pullback morphism $f^! : A^\dagger_k(Y) \to A^\dagger_{k+d} (X)$.
    \item If $f: X \to Y$ is a strict morphism of log schemes with $\ul{X} \to \ul{Y}$ a vector bundle of rank $d$ then $f^*$ is an isomorphism.
    \item If $i: D \to X$ is the strict inclusion of a closed subscheme and $j:U \to X$ the strict inclusion of the open complement then the following is an exact sequence, called the \emph{excision sequence}:
\[ A^\dagger_k (D) \xrightarrow{i_*} A^\dagger_k (X) \xrightarrow{j^*} A^\dagger_k (U) \to 0 \]
  \end{enumerate}
  
\end{thm}

These groups agree with the idea given in section 9 of~\cite{MultiplicativityOfTheDoubleRamificationCycle} for schemes log smooth over the trivial log point. As far as the definition exists one could view this as a Chow theory for the valuativisation of~\cite{KKato}. From these one can construct a theory of bivariant classes, following~\cite{Fulton}. Let $f: X \to Y$ be a log morphism, then the \emph{group of $k$ dimensional bivariant classes} $A^k_\dagger(X \xrightarrow{f} Y)$ consists of compatible collections of maps $A_* ^\dagger (Z) \to A_{*-k} ^\dagger (Z \times_Y X)$ for all morphisms $Z \to Y$, with $Z \times_Y X$ the (fs) log fibre product. We will show the following Poincar\'e duality statements, firstly one can identify the theory relative to a point with cycle classes:

\begin{thm}[Theorem~\ref{thm:classes}]

  Let $X$ be a log scheme. Then the group $A_\dagger ^* (X \to \Spec k)$ is canonically isomorphic to $ A^\dagger_* (X)$.

\end{thm}

\noindent and secondly, one has a duality pairing:

\begin{thm}[Corollary~\ref{cor:Poincare}]
  
  Let $g: Y \to Z$ be a log smooth morphism of relative dimension $k$ and $f: X \to Y$ any morphism. Then the composition map $A_\dagger ^* (X \to Y) \to A_\dagger ^{*-k} (X \to Z)$ is an isomorphism.

\end{thm}

One important application of this is to describe the bivariant theory relative to the \emph{Artin fan} of a log scheme $X$. The Artin fan $\cA_X$ of a log scheme $X$ was introduced in~\cite{BoundednessOfTheSpaceOfStableLogarithmicMaps} Section 3 and has the property that it is log \'etale over the trivial log point, and has a strict morphism $X \to \cA_X$. Since it is log \'etale over a point we can apply the above theorem to deduce the following corollary:

\begin{cor}[Example~\ref{cor:combinatorialstructure}]

  There are canonical isomorphisms
  \[A_* ^\dagger (X) \cong A^* _\dagger (X \to \cA_X) \text{\quad and \quad} A^* _\dagger (\cA_X \to \Spec k) \cong A^* _\dagger (\cA_X \to \cA_X)\]
  This gives the group $A_* ^\dagger (X)$ the structure of a module over the Chow cohomology ring $A^* _\dagger (\cA_X \to \cA_X)$.
  
\end{cor}

\subsection{Log normal cone constructions}

The ring $A^* _\dagger (\cA_X \to \cA_X)$ contains combinatorial data about the log structure on $X$, and correctly described should be a generalisation of the polytope algebra of McMullen, following~\cite{FultonSturmfels}.

Elements of $A ^* (X \to Y)$ are constructed via finding an embedding of the normal cone $\cC_{X/Y}$ into a vector bundle stack $\EE_X$ using tools developed by Manolache in~\cite{VirtualPullbacks} and~\cite{VirtualPushforward}. The log normal cone $\cC^\ell_{\bullet/\bullet}$ was constructed by Herr in~\cite{LeosPaper} and one can view it as a log scheme over $X$ by insisting that map to $X$ be strict. This is the natural extension of the normal cone to log geometry constructed from the log cotangent complex $\Log{\bullet/\bullet}$, described by~\cite{LogarithmicGeometryAndAlgebraicStacks}. Unfortunately the log cotangent complex is not as well behaved as the cotangent complex, this presents itself here as the fact that if $\pi: \tilde{X} \to X$ is a log \'etale morphism and $X \to Y$ is any morphism then $\pi^* (\cC^\ell_{X/Y}) \not \cong \cC^\ell_{\tilde{X}/Y}$, it is only a substack.

One can view this as the fact that $\pi$ is not itself flat and so poorly suited to taking fibre products. Correcting this on the level of log schemes requires a foundation of derived log schemes and Chow groups for these, a formidable challenge. Instead one could correct it on the level of cycles. We do this using the Gysin maps constructed above, combined with~\cite{LeosPaper} Remark 2.14.

\begin{thm}[Theorem~\ref{thm:normalcone}]
  
  Let $\pi:\tilde{X} \to X$ be a log \'etale map with $\tilde{X}$ and $X$ locally free, and $f: X \to Y$ a log morphism with $Y$ log smooth over the trivial log point. Then let $i: \cC^\ell_{\tilde{X}/Y} \to \pi^* \cC^\ell_{X/Y}$ be the inclusion. There is an equality $i_* [\cC^\ell_{\tilde{X}/Y}] = \pi^![\cC^\ell_{X/Y}]$.
  
\end{thm}

This produces the following corollary:

\begin{cor}[Corollary~\ref{cor:conedefined}]

  The virtual fundamental class construction applied to $\cC^\dagger _{X/Y}$ produces a well defined class in $A_* ^\dagger (X)$.
  
\end{cor}

We use this work finally to both recast Proposition 3.6.2 of~\cite{LogarithmicGromovWittenTheoryWithExpansions} in terms of these Gysin maps and to extend the construction of log Gysin maps found in Section 3 of~\cite{LeosPaper} substantially.

\subsection{Acknowledgements}

The initial part of this work was carried out as a PhD student in Cambridge under Mark Gross, and continued both as a Postdoctoral Researcher at NCTS Taipei and a Visiting Assistant Proffesor in Boston College. I must thank Mark Gross for introducing me to log geometry and teaching me the details of the subject. The theory was streamlined through discussion between myself and Leo Herr, and I am greatly indebted to him. I also owe a debt of gratitude to David Holmes for discussing these ideas with me in Leiden, and Michel van Garrel for organising a wonderful conference at MFO Oberwolfach where I discussed some of these ideas with him, Dhruv Ranganathan and Jonathan Wise. Navid Nabijou has been a constant source of inspiration. Finally many intimate details of log structures were unveiled to me by Dan Abramovich.

\section{The cycle group}

Given a log scheme $X$ we can at a first attempt just take the Chow group of $\ul{X}$. Given a modification $\tilde{X} \to X$ with both $\ul{\tilde{X}}$ and $\ul{X}$ smooth we can follow Shokurov in~\cite{Shokurov} and Holmes-Pixton-Schmidt in~\cite{MultiplicativityOfTheDoubleRamificationCycle} and construct compatible Gysin maps by taking Gysin pullbacks. We begin with a motivating lemma from~\cite{Niziol}:

\begin{lem}

  Let $X$ be a log scheme log smooth over the trivial log point. Then $\ul{X}$ is smooth if and only if the stalks of the characteristic sheaf are isomorphic to $\NN ^r$.
  
\end{lem}

\begin{proof}

  Suppose that $\ul{X}$ is smooth. Then we have a chart for the log structure over an \'etale open subset $U \subset X$ with at least one point mapping to the origin of $\Spec k[P]$:
  \begin{figure}[h]
    \centering
    \begin{tikzcd}
      U \arrow{rd}{ch} \arrow[bend left=20]{rrd}{ch} & & \\
      & \Spec k[P] \arrow{r} \arrow{d} & \Spec k[P] \arrow{d} \\
      & \Spec k \arrow{r} & \Spec k
    \end{tikzcd}
  \end{figure}

  \noindent by assumption $U$ is smooth and the map $ch: U \to \Spec k[P]$ is smooth. Taking the cotangent exact sequence:
  \[ 0 \to ch^* \Omega_{\Spec k[P]/\Spec k} \to \Omega_{U/\Spec k} \to \Omega_{U/\Spec k[P]} \to 0 \]
  we see that $\Omega_{\Spec k[P]/\Spec k}$ is free in a neighbourhood of the origin. But this implies that $P \cong \NN^r$ is free.

  Conversely if the stalks of the characteristic sheaf of $X$ are free then we have an \'etale cover of $\ul{X}$ by smooth schemes. Therefore $\ul{X}$ is smooth.

\end{proof}

This will be enough to construct the desired Gysin maps, having locally free stalks as above acts to specify compatible maps to smooth schemes. Let us notate such schemes. 

\begin{dfn}

  Let $X$ be a log scheme such that at every point the stalk of the characteristic sheaf is isomorphic to $\NN^r$, where $r$ may vary from point to point. We call such an $X$ a \emph{locally free} log scheme.
  
\end{dfn}

As a word of warning, in the sense of cones such schemes would be called ``smooth'', but the phrase ``smooth log scheme'' is evidently indicative of something else, while free cones also includes simplicial cones. I decided to stick to this direct description of the monoids. Our first step is to construct Gysin maps for certain morphisms of log schemes, locally free blow-ups morphisms of~\cite{FKato} exposited in the introduction.

\begin{dfn}

  A \emph{log blow-up} $\pi: \tilde{X} \to X$ is the morphism induced by blowing up a sheaf of log ideals on $X$. When it is finite it is degree one, in particular if $X$ is log smooth over the trivial log point $\pi$ is birational. We call a log blow-up $\pi: \tilde{X} \to X$ with $\tilde{X}$ and $X$ locally free a \emph{locally free log blow-up}.

  A \emph{log modification} $\pi: \tilde{X} \to X$ is a log morphism for which there exists a log blow-up $Z \to \tilde{X}$ such that the morphism $Z \to X$ is also a log blow-up.
  
\end{dfn}

By Theorem 3.16 of~\cite{FKato} any morphism $f: X \to Y$ may be lifted to an integral morphism $\tilde{f} : \tilde{X} \to \tilde{Y}$ after passing to log modifications $\tilde{X} \to X$ and $\tilde{Y} \to Y$. In fact an inspection of the proof shows that this can be done with $\tilde{Y} \to Y$ a log blow-up.

Studying only log blow-ups is sufficient for constructing colimits over all log modifications, they are by definition cofinal within the system of log modifications. Since they are degree one when finite they are good analogues of the birational maps used in the $b$-Chow definition.

\begin{con}

  Let $\pi: \tilde{X} \to X$ be a locally free log blow-up. We claim that this has a canonical perfect obstruction theory of relative dimension 0. Let $U \to X$ be an \'etale open set around $x \in X$ on which we have a chart for the log structure. After potentially restricting $U$ we have that this chart is of the form $U \to \AA^n$, where $n$ is the rank of $\CharSheaf{X}$ at $x$. The morphism $\pi$ is the fibre over a blow-up $\pi_{\AA^n}: Bl_I (\AA^n) \to \AA^n$, with $Bl_I (\AA^n)$ also smooth. Therefore $\pi_{\AA^n}$ is lci and has a canonical rank zero perfect obstruction theory.

  We now glue these local models to an \'etale atlas for the obstruction theory. To be precise suppose that $U$ is an \'etale open set with a chart $\AA^n$ and $V$ is a smaller \'etale open set with a chart $\AA^m$. There is an induced generisation map on charts of the form $\phi: \AA^n \to \AA^m$, given by projection to a subspace, together with an induced map on the blow ups, $\phi_{Bl}: Bl_I (\AA^n) \to Bl_I (\AA^m)$. But over the image of $V$ the map $\phi_{Bl}$ is the pullback of $\phi$. Since $\phi$ is smooth the two obstruction theories glue canonically.
  
  This therefore defines a vector bundle stack on the \'etale topology of $X$. I suspect that this descends to the Zariski topology, but in any case by~\cite{ChangLi} Theorem/Definition 3.7 there is a well defined theory of \'etale perfect obstruction theories. We write the associated Gysin map $\pi^!$ which is in degree zero, it preserves dimension.
  
\end{con}

\begin{rmk}

  When $\tilde{X}$ and $X$ are locally free and log smooth over the trivial log point then $\ul{\tilde{X}}$ and $\ul{X}$ are also smooth. Therefore $\ul{\pi}$ induces a Gysin map agreeing with the above construction. Thus these are an extension of the groups constructed in~\cite{MultiplicativityOfTheDoubleRamificationCycle}, restricted only to log blow-ups.
  
\end{rmk}

We need to check that this construction is functorial under further blow-ups.

\begin{lem}

  Let $\pi': X'' \to X'$ and $\pi: X' \to X$ be locally free log blow-ups. Then $(\pi ' \cdot \pi) ^! = \pi ^{'!} \cdot \pi^!$.
  
\end{lem}

\begin{proof}

  After passing to \'etale open sets we may assume that we have charts for the log structures, $Bl_J \AA^n \to Bl_I \AA^n \to \AA^n$. Our claim is equivalent to the compatibility of the obstruction theories for this sequence of maps, and they are evidently compatible.
  
\end{proof}

We can then define the Chow groups analogously to~\cite{MultiplicativityOfTheDoubleRamificationCycle} as being a colimit over all log blow-ups. 

\begin{dfn}

  Let $X$ be a log scheme. The category $\catname{Fr Bl}_X$ is the (non-empty) category whose objects are log blow-ups $\pi: \tilde{X} \to X$ with $\tilde{X}$ locally free and whose morphisms are further log blow-ups over $X$. The above construction gives a contravariant functor $A_* (-): \catname{Fr Bl}_X \to \catname{Gr Ab}$. The \emph{log Chow groups of $X$}, $A_*^ \dagger(X)$, are the colimit over $A_* (-)$.
  
\end{dfn}

\begin{rmk}

  As in~\cite{MultiplicativityOfTheDoubleRamificationCycle} the category of log blow-ups is filtered, and these groups may be described in terms of cycles on some blow-up, modulo identification under (Gysin) pullback.

  Using~\ref{cor:pushforward} and~\ref{NoetherianInduction} there is a canonical map from this to the inverse limit over all log blow-ups, with transition maps given by pushforward. This is the analogue of the embedding of the small $b$-Chow ring into the large $b$-Chow group. This inverse limit in the context of logarithmic geometry was suggested as an object of study by~\cite{LogarithmicGromovWittenTheoryWithExpansions} in Proposition 3.6.2. We will see that relations in our groups imply the relations there. 
  
\end{rmk}

Calculating these colimits is at least as hard as calculating Chow groups. There are two examples where we can directly describe the groups, for toric varieties described in~\cite{FultonSturmfels} Theorem 4.2 and the discussion before it, applying Poincar\'e duality and cofinality of smooth blow-ups to move from the Chow cohomology groups used there to the Chow groups used here. The authors describe these groups as the polytope algebra of McMullen,~\cite{McMullen}. The second case is when the log structure on $X$ is Deligne--Faltings rank one, in the notation of~\cite{Chen}. Here the category $\catname{Fr Bl}_X$ is the singleton $X$, with only the identity map.

Let us calculate a non-trivial example of these Gysin maps and see that there are some subtleties:

\begin{eg}

  Let $\Spec k_{\NN^2}$ be the origin in $\AA^2$ with its induced log structure. The blow up at the origin induces a locally free log blow-up $\PP^1 \to \Spec k_{\NN^2}$ where $\PP^1$ has globally non-trivial log structure. The only possible class in dimension zero is again a multiple of a point. Calculating the normal cone we see that we are intersecting two copies of the zero section inside $\cO(1)$ on $\PP^1$, so the multiplicity is 1.

\end{eg}

There is a problem with this, unless $X$ is log flat over the trivial log point formation of the fundamental class does not commute with this pullback. Under sufficiently nice conditions however this does commute. Our first step is to have some notation for when all the log blow-ups of a given log scheme $X$ have the same dimension.

\begin{dfn}

  Let $X$ be a log scheme such that for the generic point of every log strata $\zeta$ the sum $ \rk \CharSheaf{\zeta} - \codim \zeta$ is constant on each irreducible component. This condition is stable under log flat morphisms and we call such a log scheme \emph{of pure log dimension}. The negation of this is \emph{not} stable under log flat morphisms.

  If in addition every generic point of a component of $X$ has log structure $0$ or $\NN$ then we say that $X$ is $\emph{maximal pure log dimension $\dim \ul{X}$}$. For every further log blow-up of $X$, $\tilde{X} \to X$ we have that $\dim \ul{\tilde{X}} = \dim \ul{X}$.

  We say that a locally free log blow-up $\pi: \tilde{X} \to X$ is \emph{compatible with the fundamental class (of $X$)} if $\pi^! [X] = [\tilde{X}]$. We say that $X$ \emph{has compatible fundamental class} if $X$ is locally free and every locally free blow-up is compatible with the fundamental class.
  
\end{dfn}

The idea is that if $X$ has maximal pure log dimension then the Gysin pullback of $[X]$ to any log blow-up is again top dimensional and so possibly equal to the fundamental class. The following lemma proves that this intuition is indeed correct.

\begin{lem}

  Let $X$ be a locally free log scheme of maximal pure log dimension. Then $X$ has compatible fundamental class.
\label{compatiblefundamentalclass}  
\end{lem}

\begin{proof}

  Let $\pi: \tilde{X} \to X$ be a locally free blow-up. Then as stated above $\dim \ul{\tilde{X}} = \dim \ul{X}$, and so $\pi^! [X]$ is a top dimensional class on $\tilde{X}$. We are free to suppose that $\tilde{X}$ is the union $\tilde{X}_1, \ldots \tilde{X}_n$, so that $\pi^! [X] = \sum n_i [\tilde{X}_i]$ and the claim is that $n_i = mult \: \tilde{X}_i$ where $mult$ denotes the length at the generic point. Let us fix one of the components of $\tilde{X}$, $\tilde{X}_1$

  Since we are dealing with top dimensional classes this is local on $X$, so we may move to \'etale open sets $\tilde{U}$ and $U$ of $\tilde{X}$ and $X$ such that $\tilde{U}$ contains the generic point of $\tilde{X}_1$ and where we have charts fitting into the following diagram:

  \begin{figure}[h]
    \centering
  \begin{tikzcd}
    \tilde{U} \arrow{r} \arrow{d} & U \arrow{d} \\
    Bl_I (\AA^n) \arrow{r} & \AA^n
  \end{tikzcd}
  \end{figure}

  There is an associated embedding of normal cones $\cC_{\pi} \to ch_{\tilde{U}} ^* \cN_{\pi_I}$, which we wish to show is an isomorphism. There is an affine bundle $\AA^n_U \to U$ with a relatively smooth blow up $\pi: Bl_I (\AA^n_U) \to \AA^n _U$ which is smooth over $U$ given by pulling back charts. Say that $U$ is cut out as the zero section by $x_1, \ldots x_n$, then $\pi^* x_1, \ldots \pi^* x_n$ cut out $\tilde{U}$. These remain regular since $\tilde{U}$ is codimension $n$ inside $Bl_I (\AA^n _U)$. In particular the normal cone of $\tilde{U} \to U$ is a bundle.

We now have a sub-bundle of the same dimension as $ch_{\tilde{U}} ^* \cN_{\pi_I}$, hence these are equal and the desired class is just $mult_i [U_i]$.

\end{proof}  

These log blow-ups are almost birational, there may be some components that are contracted but away from these the map is birational. A corollary is:

\begin{cor}
\label{cor:pushforward}
Suppose that $X$ has compatible fundamental class and $\pi: \tilde{X} \to X$ a locally free blow-up, then $\pi_* \pi^! ([X]) = [X]$.

\end{cor}

From this lemma we can deduce another corollary (the proof is obvious in the quasi-projective case using the moving lemma and the above result, but it remains true in the general case):

\begin{cor}
\label{NoetherianInduction}
  Let $Y$ be a log scheme. Then there is a blow-up $\pi: \tilde{Y} \to Y$ and a collection of closed embeddings $i_j: Z_j \to \tilde{Y}$ and integers $n_j$ such that each $Z_j$ has compatible fundamental class and $\pi^! [Y] = \sum n_j i_* [Z_j]$.
  
\end{cor}

\begin{proof}

  We prove this by Noetherian induction on the dimension of the images of various strata in $Y$. Let $Y_1 \to Y$ be the barycentric subdivision of $Y$. This is of maximal pure log dimension and every irreducible component has generic log structure $\NN$. If $Y$ is quasi-projective then we can apply the moving lemma~\cite{StacksMovingLemma} to $\pi^! ([Y])$ to get a transverse representative. We apply this idea, together with Noetherian induction, to prove the claim for all $Y$.

  Choose an open affine of $Y_1$, $U_1$, with complement $D_1$. Since $U_1$ is affine it is quasi-projective and we represent $\pi^![U_1]$ by a collection of classes $\sum n_i Z_1^{(i)}$ which are transverse to all the log strata. The closure of these need not be transverse, but we can refine $Y_1$ further until they are, changing $U_1$ by log blow-ups. Since the representatives have compatible fundamental class on $U_1$, they retain this after a further blow-up of $U_1$. The difference between $\pi^! ([Y])$ and $\sum n_i  \bar{Z_1^{(i)}}$ is supported over $D_1$. Apply the same argument to these classes over an open affine $U_2$ of $D_1$, the fibre over $U_2$ of any log blow-up is again quasi-projective, with complement $D_2$ and extend the blow-ups as needed over $U_1$.

  Therefore by Noetherian induction we reduce to the case where $D_n$ is a point inside $Y$. Then we have a collection of classes on $\PP^m$ which we want to write as being transverse to the log structure. But these are represented by a collection of hyperplanes.
  
\end{proof}

\subsection{Actions of morphisms}

We now must construct the usual paraphenalia of Chow theory, a log flat pullback, a proper pushforward, Gysin maps etc. We first demonstrate that a little care must be taken with log flat pullback and proper pushforward.

\begin{eg}

  Consider the square:
  
  \begin{figure}[H]
    \centering
    \begin{tikzcd}
      \PP^1 \arrow{r}{\pi} \arrow{d}{\pi} & \Spec k_{\NN^2} \arrow{d} \\ \Spec k_{\NN^2} \arrow{r} & \Spec k_{\NN^2}
      \end{tikzcd}
  \end{figure}
  \noindent
where $\PP^1 \to \Spec k_{\NN^2}$ is the morphism in the previous example and the other two morphisms are the identity. We compare $\pi^*$ and $\pi^!$ applied to the class of a point to see that in this case they must agree. In particular here the log flat pullback is \emph{not} given by taking inverse image, even though all of the maps are also flat.

Now consider the square:
  \begin{figure}[h]
    \centering
    \begin{tikzcd}
      \PP^1 \arrow{r}{\pi} \arrow{d}{\pi} & \PP^1 \arrow{d}{\one} \\ \PP^1 \arrow{r}{\one} & \Spec k_{\NN^2}
      \end{tikzcd}
  \end{figure}
  
\noindent  this square is Cartesian in the category of fs log schemes. We compare pushforward of the fundamental class of $\PP^1$, working both ways around this square and starting in the bottom left. The composite $\one_* \one^! ([\PP^1])$ is of course $[\PP^1]$. Unfortunately the composite $\pi^! \pi_* ([\PP^1])$ is zero, since $\pi_* ([\PP^1])$ is already zero.

  In both cases the problem stems from the fact that neither class is represented by a cycle with compatible fundamental class.
  
\end{eg}

\begin{con}
\label{con:pushforward}
    Let $f: X \to Y$ be a proper morphism with $X$ and $Y$ locally free. Suppose that we have $\tilde{f}: \tilde{X} \to \tilde{Y}$ an integral lift of $f$, with log blow-up $\pi_Y$ and modification $\pi_X$ respectively, $\tilde{Y}$ locally free and $\bar{X} \to \tilde{X}$ a log blow-up with $\bar{X}$ locally free. In diagrams we have the following picture:
  \begin{figure}[h]
    \centering
    \begin{tikzcd} \bar{X} \arrow{r}{\psi_{\tilde{X}}} \arrow{rd}{\phi_X} & \tilde{X} \arrow{d}{\pi_X} \arrow{r}{\tilde{f}}  & \tilde{Y} \arrow{d}{\pi_Y} \\ & X \arrow{r}{f} & Y 
      \end{tikzcd}
  \end{figure}
  
  There is a canonical map from $A_* (X)$ to $A_* (\tilde{Y})$ defined by $\tilde{f}_* \psi_{\tilde{X} \: *} \phi^!_X$, which we denote $f_\# ^{\tilde{Y}}$. This notation already supresses various of the dependencies of this map, for instance on the choice of $\bar{X}$. We first show independence from the choice of integral lift after passing to a common refinement using the Gysin maps.

  Any two choices of integral lift are dominated by a third, so let us assume that we have $\hat{f} : \hat{X} \to \hat{Y}$ a integral lift of $\tilde{f}$ (which is already itself integral!), with log blow-ups $\pi_{\tilde{Y}}$ and log modifications $\pi_{\tilde{X}}$ respectively, $\hat{Y}$ locally free, and $\psi_{\hat{X}}: \check {X} \to \hat{X}$ a further log blow-up such that $\check{X}$ is locally free and admits a log blow-up map $\phi_{\bar{X}}:\check{X} \to \bar{X}$. In short we find ourselves with this diagram: 
  \begin{figure}[h]
    \centering
    \begin{tikzcd}
       \check {X} \arrow{r}{\psi_{\hat{X}}} \arrow{rd}{\phi_{\bar{X}}} & \hat{X} \arrow[bend left=30]{rr}{\hat{f}} \arrow{r} \arrow{d}{\pi_{\bar{X}}} & \tilde{X} \times_{\tilde{Y}} \hat{Y} \arrow{r}{\tilde{f}{\mid}_{\hat{Y}}} \arrow{d}{\pi_{\tilde{Y}} {\mid} _{\tilde{X}}} & \hat{Y} \arrow{d}{\pi_{\tilde{Y}}} \\ & \bar{X} \arrow{r}{\psi_{\tilde{X}}} \arrow{rd}{\phi_X} & \tilde{X} \arrow{d}{\pi_X} \arrow{r}{\tilde{f}}  & \tilde{Y} \arrow{d}{\pi_Y} \\ & & X \arrow{r}{f} & Y 
      \end{tikzcd}
  \end{figure}

  \noindent where since $\tilde{f}$ is integral the underlying scheme of $\tilde{X} \times_{\tilde{Y}} \hat{Y}$ is the product of the underlying schemes. Suppose that $[V]$ is a class on $X$ representing a strict embedding $V \to X$ where $V$ has compatible fundamental class and we wish to show that $\pi_{\tilde{Y}}^! f_\# ^{\tilde{Y}} ([V]) = f_\# ^{\hat{Y}} ([V])$. After replacing $X$ by $V$ we may assume that $X$ has compatible fundamental class. Now we wish to prove that $\pi_{\tilde{Y}}^! f_\# ^{\tilde{Y}} ([X]) = f_\# ^{\hat{Y}} ([X])$. First note that since $\tilde{f}$ is integral the morphism $\pi_{\tilde{Y}}{\mid}_{\tilde{X}}: \tilde{X} \times_{\tilde{Y}} \hat{Y} \to \tilde{X}$ has an induced perfect obstruction theory and $\pi_{\tilde{Y}}{\mid}_{\tilde{X}}^! ([\tilde{X}]) = [\tilde{X} \times_{\tilde{Y}} \hat{Y}]$ by the same arguments as in~\ref{compatiblefundamentalclass}. Then by definition:
  \begin{align*}
    \pi_{\tilde{Y}}^! f_\# ^{\tilde{Y}} ([X]) &= \pi_{\tilde{Y}}^! \tilde{f}_* \psi_{\tilde{X} \: *} \phi^!_X ([X]) \\
    &= \pi_{\tilde{Y}}^! \tilde{f}_* ([\tilde{X}])\\
    &= (\tilde{f} {\mid}_{\hat{Y}})_* (\pi_{\tilde{Y}}{\mid}_{\tilde{X}})^! ([\tilde{X}]) \\
    &= (\tilde{f} {\mid}_{\hat{Y}})_* ([\tilde{X} \times_{\tilde{Y}} \hat{Y}]) \\
    &=  \hat{f}_* \psi_{\hat{X} \: *} ([\check{X}]) \\
    &= f_\# ^{\hat{Y}} ([X])
    \end{align*}
\noindent therefore this induces a well-defined map on the cycle classes on $X$ which are represented by cycles with compatible fundamental class.
  
  Now suppose that $\hat{f}: \hat{X} \to \hat{Y}$ is a further integral lift of $\tilde{f} \psi_{\tilde{X}}$ with blow-ups $\pi_{\bar{X}}$ and $\pi_{\tilde{Y}}$ respectively and $\psi_{\hat{X}}: \check {X} \to \hat{X}$ is a blow-up with $\check{X}$ locally free and admitting a blow-up map $\phi_{\bar{X}}: \check{X} \to \bar{X}$. In short we have the following picture:
  \begin{figure}[h]
    \centering
    \begin{tikzcd}
       \check {X} \arrow{r}{\psi_{\hat{X}}} \arrow{rd}{\phi_{\bar{X}}} & \hat{X} \arrow{rr}{\hat{f}} \arrow{d}{\pi_{\bar{X}}} & \:  & \hat{Y} \arrow{d}{\pi_{\tilde{Y}}} \\ & \bar{X} \arrow{r}{\psi_{\tilde{X}}} \arrow{rd}{\phi_X} & \tilde{X} \arrow{d}{\pi_X} \arrow{r}{\tilde{f}}  & \tilde{Y} \arrow{d}{\pi_Y} \\ & & X \arrow{r}{f} & Y 
      \end{tikzcd}
  \end{figure}

\noindent  and we wish to compare $(\tilde{f} \psi_{\tilde{X}})_\#^{\hat{Y}} \phi_X ^! ([V])$ and $f_\#^{\hat{Y}} ([V])$ for some cycle class $[V]$ where $V$ has compatible fundamental class. Then as before we may replace $X$ by $V$. By direct comparison of classes on $\check{X}$ we have an equality $(\tilde{f} \psi_{\tilde{X}})_\#^{\hat{Y}} \phi_X ^! ([X]) = f_\#^{\hat{Y}} ([X])$. 

  Therefore we get a well defined map on the level of log Chow groups $f_*: A_* ^\dagger (X) \to A_* ^\dagger (Y)$ by representing a cycle class by a sum of cycles with compatible fundamental class. If $X$ and $Y$ are not locally free then we may pass to a blow-up of $X$ and $Y$ which are.

  \end{con}
  
We can now construct a log flat pullback map.

\begin{con}
\label{con:pullback}
  Let $f: X \to Y$ be a log flat morphism with $X$ and $Y$ locally free. Let $\tilde{f}: \tilde{X} \to \tilde{Y}$ be a log flat and integral lift of $f$ with log blow-ups $\pi_Y$ and log modifications $\pi_X$ respectively, and $\psi_{\tilde{X}} :\bar{X} \to \tilde{X}$ a blow-up with $\bar{X}$ locally free and write $\phi_{X}: \bar{X} \to X$ for the induced blow-up of $X$. As before we obtain a diagram looking as follows: 
  \begin{figure}[h]
    \centering
    \begin{tikzcd}
\tilde{X} \arrow{d}{\pi_X} \arrow{r}{\tilde{f}}  & \tilde{Y} \arrow{d}{\pi_Y} \\ X \arrow{r}{f} & Y 
      \end{tikzcd}
  \end{figure}

\noindent  The map $\tilde{f}$ is log flat and integral, hence has flat underlying morphism. Therefore there is a canonical map $A_* (Y) \to A_*(\bar{X})$ given by $\pi_{X\: *} \tilde{f}^* \pi_Y ^!$ and we write this $f^\#_{\bar{X}}$. We note that if $[V]$ is a class on $Y$ induced by a strict embedding $V \to Y$ such that both $V$ and $f^{-1} V$ have compatible fundamental classes then $f^\# ([V]) = [X \times_Y V]$.

  As before we prove that this is independent of the choice of $\tilde{f}$, and compatible with the Gysin maps. First suppose that we had a further choice of integral lift of $\tilde{f}$ (which is of course already integral). Then we have a diagram like the following:
  \begin{figure}[h]
    \centering
    \begin{tikzcd}
\bar{X} \arrow{d}{\pi_{\tilde{X}}} \arrow{r}{\bar{f}} & \bar{Y} \arrow{d}{\pi_{\tilde{Y}}} \\ \tilde{X} \arrow{d}{\pi_X} \arrow{r}{\tilde{f}}  & \tilde{Y} \arrow{d}{\pi_Y} \\ X \arrow{r}{f} & Y 
      \end{tikzcd}
  \end{figure}

  So long as $[V]$ has compatible fundamental class the above remark shows that $f^\# ([V])$ is independent of the choice of integralisation.

  Now suppose that we have an integralisation $\tilde{f} : \tilde{X} \to \tilde{Y}$ with $\tilde{Y}$ locally free and a further blow-up $\pi_X : \check{X} \to X$ admitting a blow-up map $\phi_{\tilde{X}}: \check{X} \to \tilde{X}$. Let $\bar{f}: \bar{X} \to \bar{Y}$ be an integralisation of $\tilde{f} \phi_{\tilde{X}}$ with $\bar{Y}$ locally free. Then we have the following diagram:

  \begin{figure}[h]
    \centering
    \begin{tikzcd}
       \bar{X} \arrow{rr} \arrow{d} & & \bar{Y} \arrow{d}{\pi_{\tilde{Y}}}\\
       \check{X} \arrow{r}{\phi_{\tilde{X}}} \arrow{rd}{\pi_X} & \tilde{X} \arrow{r}{\tilde{f}} \arrow{d} & \tilde{Y} \arrow{d}{\pi_Y} \\
       & X \arrow{r}{f} & Y 
      \end{tikzcd}
\end{figure}

  \noindent where we wish to show the equalities $\pi_X ^! f^\# (\alpha) = (\phi_{\tilde{X}} \tilde{f})^\# \pi_Y ^! (\alpha)$. If $\alpha$ is the class of a cycle $V$ with compatible fundamental class these hold by the above formula.

  Therefore this extends to a morphism $f^*: A_* ^\dagger (Y) \to A_* ^\dagger(X)$ since every cycle is equivalent to one which has compatible fundamental class after aprropriate blow-ups and we may further refine the source to ensure that the inverse image has compatible fundamental class as well. If $X$ and $Y$ are not locally free then we may pass to a blow-up which is.

This contruction is compatible with fibre products. Given $f: X \to Y$ a locally free log flat morphism and $g: V \to Y$ any morphism then $f': X \times_Y V \to V$ is also log flat, however $X \times_Y V$ need not also be locally free. Nonetheless we can choose a refinement of $X \times_Y V$ to a log scheme which is locally free, and with the induced map to $V$ still log flat. Given a cycle on $V$ with compatible fundamental class the pullback is still given by taking the inverse image.  
  
\end{con}

Fulton constructs one more tool, Gysin maps.

\begin{con}

Suppose that $D \to X$ is the strict inclusion of a regularly embedded subscheme. If $D$ has compatible fundamental class then for any blow-up $\pi_X: \tilde{X} \to X$ the embedding $\tilde{D} := D \times_X \tilde{X} \to \tilde{X}$ remains a regularly embedded subscheme and the induced Gysin maps are compatible. Therefore they induce a morphism $A_* ^\dagger (X) \to A_* ^\dagger (D)$. We call such a map a \emph{strict Gysin map}.

\end{con}

For strict closed and open embeddings these constructions simplify considerably and give us an excision sequence. Similarly for strict affine or projective bundles we can apply the same proofs as in~\cite{Fulton}.

\begin{cor}
\label{cor:excision}
  Let $D$ be a strict closed subscheme of $X$ and $U$ the complement with strict log structure. Then there is an exact sequence:

  \[ A_* ^\dagger (D) \to A_* ^\dagger (X) \to A_* ^\dagger (U) \to 0\]

  If $X \to B$ is a strict rank $k$ affine bundle then log flat pullback gives an isomorphism $A_* ^\dagger (B) \to A_{*+k}^\dagger (X)$. Similarly if $X \to B$ is a strict rank $k$ projective bundle there is an isomorphism $A_* ^\dagger (X) \cong \bigoplus A_{*+i} ^{\dagger} (B) \xi^i$.
  
  \end{cor}

Of course there are various relations between the different maps, functoriality etc which guarantee that we have canonical orientations. These are not immediate since both definitions mix pushforward, pullback and Gysin maps in non-trivial ways.

\begin{thm}

  These constructions are functorial and commute with one another in Cartesian squares in the fine log category restricting to proper pushforward along saturated morphisms.
  
\end{thm}

\begin{proof}

  Functoriality of log flat pull back follows from the description as a Cartesian product for classes with compatible fundamental class. Functoriality of proper pushforward follows by functoriality of pushforward for the classical pushforward. Functoriality for strict Gysin maps follows directly from the classical statement, since strict fs fibre products have underlying scheme the product of underlying schemes.

  To prove commutativity we suppose that we have a Cartesian diagram in the fs category:
  \begin{figure}[H]
    \centering
    \begin{tikzcd}
      W \arrow{r}{f'} \arrow{d}{g'} & V \arrow{d}{g} \\ X \arrow{r}{f} & Y
  \end{tikzcd}
  \end{figure}

  \noindent where $X$, $Y$ and $Z$ are locally free and let $\hat{W}$ be a locally free resolution of $W$. Suppose that f is a strict regular embedding and $g$ is proper or log flat. Since exactness and integrality are preserved under fs base change, strict fibre products commute with passing to underlying schemes and Gysin maps we can find compatible diagrams for the pullback or pushforward along $g$ and the equality $g'^* f^! = f'^! g^*$ or $f^! g_* = g'_* f'^!$ holds.

  Now suppose that $f$ is proper and $g$ log flat and we wish to show that $g^* f_* [Z] = g'^* f'_* [Z]$ where $Z \subset X$ has compatible fundamental class. By taking the base change along $Z \to X$ we may assume that $Z = X$, and after refining further we may assume that $X$, $W$ and the image of $X$ have compatible fundamental class.

  To calculate $f_* [X]$ we choose an integralisation $\tilde{f}: \tilde{X} \to \tilde{Y}$, and a locally free blow-up $\mathring{X}$. We would end with a class on $\tilde{Y}$ which we may assume is locally free. Let $\tilde{V}$ be a locally free blow-up of $V$ admitting a map, $\tilde{g}$, to $\tilde{Y}$. To calculate $\tilde{g}^*$ we choose an integralisation $\bar{g}: \bar{V} \to \bar{Y}$. Let $\bar{X}$ be the induced blow-up of $\tilde{X}$, and note that we are free to assume that $\mathring{X}$ maps to this. Write $\bar{W}$ for the induced fibre product. In short we have the following diagram with Cartesian squares: 
  \begin{figure}[h]
    \centering
    \begin{tikzcd}
      & & & & \bar{V} \arrow{rd}{\psi_{\tilde{V}}} \arrow{ddd}{\bar{g}} & & \\
      & & & & & \tilde{V} \arrow{rd} \arrow{ddd}{\tilde{g}} & \\
      & & & & & & V \arrow{ddd} \\
      & \bar{X} \arrow{rd} \arrow{rrr}{\bar{f}} & & & \bar{Y} \arrow{rd}{\psi_{\tilde{Y}}} & & \\
      \mathring{X} \arrow{ru}{\pi_{\bar{X}}} \arrow{rr}{\pi_{\tilde{X}}} \arrow{rrrd}{\pi_X} & & \tilde{X} \arrow{rd} \arrow{rrr}{\tilde{f}} & & & \tilde{Y} \arrow{rd}{\psi_Y} & \\
      & & & X \arrow{rrr}{f} & & & Y \\
    \end{tikzcd}
  \end{figure}
  
  \noindent The composite we wish to calculate is
  \[\pi_{\tilde{V} \: *} \bar{g}^* \psi_{\tilde{Y}}^! \tilde{f}_* \pi_{\tilde{X} \: *}\pi_X^! ([X])\]
  By the above this is equal to
  \[\pi_{\tilde{V} \: *} \bar{g}^* \bar{f}_* \pi_{\bar{X} \: *}\pi_X^! ([X])\]
  and since $X$ has compatible fundamental class this is
  \[\pi_{\tilde{V} \: *} \bar{g}^* \bar{f}_* ([\bar{X}])\]
  
Meanwhile to calculate $g'^*$ we must choose an integralisation of the induced map $\hat{W} \to X$, which we may assume lives over $\mathring{X}$, so $\mathring{W} \to \mathring{X}$. This produces a class on $\hat{W}$ and we wish to push this forward to a class on $V$. To do so we integralise this morphism, which we may assume occurs over $\tilde{V}$, so we have $\tilde{W} \to \tilde{V}$ and a locally free blow-up of $\tilde{W}$, $\check{W}$. In short we have the following diagram:
  \begin{figure}[h]
    \centering
    \begin{tikzcd}
      \check{W} \arrow{rr}{\psi_{\tilde{W}}} \arrow{rd}{\psi_{\hat{W}}} \arrow[bend right=30]{rrdd}{\psi_{\hat{W}}} & & \tilde{W} \arrow{dd} \arrow{rrrdd}{\tilde{f}'} & & & & \\
      & \bar{W} \arrow{rrr}{\bar{f}'} \arrow{ddd}{\bar{g}'} \arrow{rd} &  & & \bar{V} \arrow{rd}{\psi_{\bar{V}}} \arrow{ddd}{\bar{g}} & & \\
      \mathring{W} \arrow{ddd}{\mathring{g}'} \arrow{rr} {\pi_{\hat{W}}} \arrow{rrrd} & & \hat{W} \arrow{rd} \arrow{rrr} & & & \tilde{V} \arrow{ddd} \arrow{rd} & \\
      & & & W \arrow{ddd}& & & V \\
      & \bar{X} \arrow{rd} \arrow{rrr}{\bar{f}} & & & \bar{Y} \arrow{rd} & & \\
      \mathring{X} \arrow{ru}{\pi_{\bar{X}}} \arrow{rr} \arrow{rrrd}{\pi_X} & & \tilde{X} \arrow{rd} \arrow{rrr} & & & \tilde{Y}  & \\
      & & & X & & & \\
    \end{tikzcd}
  \end{figure}  

  \noindent and we wish to calculate
  \[\tilde{f}'_* \psi_{\tilde{W}} \psi_{\hat{W}}^! \pi_{\hat{W} \: *} \mathring{g}'^* \pi_X ^! [X]\]
  This is equal to
  \[\psi_{\tilde{V} \: *}\bar{f}'_* \psi_{\bar{W}} \psi_{\hat{W}}^! \pi_{\hat{W} \: *} \mathring{g}'^* \pi_X ^! [X]\]
  Since $X$ and $W$ have compatible fundamental class this is equal to $\psi_{\tilde{V} \: *}\bar{f}'_* [\bar{W}]$, or expanding this via functoriality:
  \[\psi_{\tilde{V} \: *}\bar{f}'_* \bar{g'}^* [\bar{X}]\]
  Finally the square with bars is Cartesian in the category of schemes and $\bar{g}$ is flat and so we can commute $\bar{f}_*$ and $\bar{g}^*$. Therefore the two expressions are equal.
    
\end{proof}

We recall the \emph{Artin fan} construction of~\cite{BoundednessOfTheSpaceOfStableLogarithmicMaps} and\cite{AbramovichWise}. This associates to a log scheme $X$ a canonical log Artin stack $\cA_X$ such that there is a strict map $X \to \cA_X$ and $\cA_X$ is log \'etale over the trivial log point. A scheme $X$ is log flat over $\Spec k$ if and only if it is flat over $\cA_X$, and similarly for log \'etale or log smooth.

So now suppose that $X$ is log flat over the trivial log point. Then we can pullback classes from $\cA_X$ via flat pullback. Given a log strata of $X$, $X_M$, of codimension $n$ there is a canonical subscheme in $\cA_X$ of codimension $n$, and flat pullback maps this class with $[X_M]$. When $X$ is a toric variety with its toric log structure the Chow groups are generated by the closure of torus orbits and hence this pullback is in fact an isomorphism. 

\subsection{Bivariant classes and Poincar\'e duality}

The most useful part of the theory is the study of bivariant classes, which we are now able to introduce:

\begin{dfn}

We define the \emph{log bivariant theory of $\psi:X \to Y$ of relative dimension $k$} to be the group $A^k_\dagger (f: X \to Y)$ generated by collections of morphisms $f_\phi$ for every map $\phi: Z \to Y$ mapping $A_*^\dagger (Z) \to A_{*-k} ^\dagger (X \times_Y Z)$ and commuting with saturated proper pushforward, log flat pullback and all strict Gysin maps.

\end{dfn}

For instance any Cartier divisor $D$ with compatible fundamental class on any blow-up of $X$ induces an element in the log Chow cohomology $A^1 (D \to X)$. Similarly a strict map between schemes log smooth over the trivial log point can be refined to one between log schemes with smooth underlying scheme, and the induced Gysin map commutes with further refinements. Once we have a version of Poincar\'e duality we will be able to drop the strictness assumption.

\begin{dfn}

  The final piece of data that we need is a product map for cycles. Given cycles $\alpha$ and $\beta$ on $X$ and $Y$ respectively we may lift these to actual closed subschemes of $X$ and $Y$ with $X$ and $Y$ locally free. Then the product $X \times_{\Spec k} Y$ is again locally free, and we have an associated cycle $\alpha \times \beta$. These depend on $\alpha$ and $\beta$ up to rational equivalence and is stable under Gysin pullback once $\alpha$ and $\beta$ have compatible fundamental class.

\end{dfn}

\begin{rmk}

In~\cite{MultiplicativityOfTheDoubleRamificationCycle} and~\cite{DhruvProducts} much of the hard work deals with the fact that intersection products do not commute with pushforward. Given a refinement $\pi: \tilde{X} \to X$ and cycles $\alpha$ and $\beta$ on $\tilde{X}$ there is \emph{not} an equality $\pi_* (\alpha \cap \beta) = \pi_*\alpha \cap \pi_*\beta$.

This does not occur here, the map $\pi_*$ sends the class $\alpha \cap \beta$ on $\tilde{X} \times \tilde{X}$ to the class $\alpha \cap \beta$ on $X \times X$, $\alpha$ to $\alpha$ on $X$ and $\beta$ to $\beta$, as colimits the groups remember the log schemes on which the cycles live, at least up to further blow-ups.

\end{rmk}

A certain amount of algebra copies across from the classical case. In particular we get the following version of \emph{Poincar\'e Duality} for these groups.

\begin{cor}
\label{cor:Poincare}
  Let $f: X \to Y$ be an arbitrary log morphism and $g: Y \to Z$ be a log smooth morphism of relative dimension $d$. Then the composition map $A^{*}_\dagger (f : X \to Y)$ to $A^{*-d}_\dagger (gf:X \to Z)$ is an isomorphism with inverse $\Gamma_f \cdot g^* (-)$.
  
\end{cor}

The distinguishing feature of classical Chow groups is the presence of an isomorphism between the bivariant theory of the structure map to a point and the cycle groups. We will prove that, although not all log schemes are log flat over the trivial log point, this result continues to hold by embedding $X$ into the enveloping vector bundle $E$.

\begin{thm}
\label{thm:classes}
  Let $X$ be a log scheme. Then there is a canonical isomorphism $A_\dagger ^{-*} (X \to \Spec k) \to A_* ^\dagger (X)$.
  
\end{thm}

\begin{proof}

For any locally free log scheme we may embed $V$ strictly as the zero section of a vector bundle $E$ defined by the $\NN$ faces of the log structure, together with sections defined up to an action of the \'etale fundamental group. From the local charts $E{\mid}_U \cong U \times \AA^n \to \AA^n$ where the identification is the one induced by the sections we see that $E$ is in fact log flat over the trivial log point.
  
Let $f$ be a bivariant class in $A_\dagger ^* (X \to \Spec k)$ and let $\alpha$ be $f (\Spec k)$, the corresponding evaluation against a fundamental class. Then let $g: V \to \Spec k$ be a log scheme with $V$ locally free, we have the following diagram:

\begin{figure}[h]
  \centering
  \begin{tikzcd}
    V \times X \arrow{r} \arrow{d} & E \times X \arrow {r} \arrow{d} & X \arrow{d} \\
    V \arrow{r}{0} & E \arrow{r} & \Spec k
    \end{tikzcd}
  \end{figure}

Since the horizontal rows are strict regular embeddings and flat morphisms the whole square has commuting Gysin maps. In particular we see that the value of $\alpha ([V])$ is $0^! ([E] \times \alpha)$. Therefore this evaluation map was an isomorphism.
  
\end{proof}

\begin{rmk}

  There is no such analogous construction over $\Spec k^\dagger$, and in general there is not a natural map $A_* (X) \to A^{-*} (X \to \Spec k^\dagger)$ splitting the one $A^{-*} (X \to \Spec k^\dagger) \to A_* (X)$. The fundamental reason is that if $X$ and $Y$ are locally free there is no reason that $X \times_{\Spec k^\dagger} Y$ need be locally free.

By~\cite{BlochGilletSoule} and applied in~\cite{vG} we can construct the Chow cohomology group of the central fibre of a strict semi-stable degeneration. This allows one to intersect with classes occuring as limits of families from the general fibre. In our setting this can be interpreted by showing that after some log blow-ups the limits of such families can be represented by cycles on the central fibre which are log flat over the standard log point. These have a well defined log flat pullback and pulling back along a strict resolution of the diagonal gives the desired product.
  
\end{rmk}

We can now apply this corollary to obtain several useful examples.

\begin{eg}
\label{eg:SmoothGysin}
  Let $X \to Y$ be a morphism with $Y$ log smooth over the trivial log point, suppose that $X$ is of pure log dimension $d_X$, $Y$ is of pure log dimension $d_Y$, and write $d = d_Y - d_X$. Then there is a Gysin pullback map in $A_\dagger ^d (X \to Y)$ corresponding to the fundamental class in $A^\dagger _{d_X} (X)$. 
  
\end{eg}

\begin{eg}
  \label{cor:combinatorialstructure}
  We apply the above Poincar\'e duality result to the map $X \to \cA_X$ for a log scheme $X$ to obtain canonical isomorphism $ A^* _\dagger (X \to \Spec k) \cong A^* _\dagger (X \to \cA_X)$.

  This makes $A_* ^\dagger (X)$ into a module over the ring $A^* _\dagger (\cA_X)$, which is in turn isomorphic to $A_{-*} ^\dagger (\cA_X)$. This is equivalent to the statement that one can intersect cycles with the closure of locally free log strata, since these strata are intersections of Cartier divisors.

  Furthermore when $X$ is a toric variety with the toric log structure the toric divisors on blow-ups of $X$ span both $A_*^\dagger (X)$ and $A_* ^\dagger (\cA_X)$. Therefore the ring $A_* ^\dagger (\cA_X)$ is also isomorphic to the McMullen polytope algebra $\Pi$. I do not know of a construction of a polytope algebra for a general Artin fan.
  
\end{eg}

Extending this to the relative setting we obtain:

\begin{eg}

  Let $f : X \to Y$ be a morphism of log stacks. The relative Artin fan $\cA_{X/Y}$ of~\cite{BoundednessOfTheSpaceOfStableLogarithmicMaps} admits a strict log \'etale map from $Y$, and hence there is an isomorphism $A_\dagger ^* (X \to Y) \cong A_\dagger^* (X \to \cA_{X/Y})$.
  
\end{eg}

\section{The log normal cone}

Given a morphism of log schemes $f: X \to Y$ Leo Herr in~\cite{LeosPaper} defines a \emph{log normal cone}, $\cC_{X/Y}^\ell$, a cone strict over $X$ of dimension equal to the dimension of $\cL og_Y$. If $f$ is strict this is simply the usual normal cone. In general one factors $f$ locally as a strict closed embedding into a smooth cover of $Y$. Equivalently this is the restriction of the cone of $\cL og_X$ over $\cL og_Y$ to $X$.

Let us give the standard example, due to W. Brauer, of how this can fail to be well behaved under log blow-ups, and how to fix it:

\begin{eg}

  Let $i: \Spec k _{\NN^2} \to \AA^2$ be the origin inside $\AA^2$ with its toric log structure. Let $\phi: Bl_0 (\AA^2) \to \AA^2$ be the blow up at the origin. The fibre product of these two morphisms defines a log structure on $\PP^1$. We write $j: \PP^1 \to Bl_0 (\AA^2)$ for the induced inclusion.

  The normal cone of $i$ is the trivial rank two bundle on $\Spec k$, while the normal cone for $j$ is $\cO(-1)$ on $\PP^1$. It is of course clear for rank reasons that $\phi^* \cC^\ell_{i} \not \cong \cC^\ell _j$, even though the log blow-ups are log flat. We can however form the following diagram:

  \begin{figure}[h]
    \centering
    \begin{tikzcd}
      \phi^* \cC^\ell_i \arrow{r} \arrow{d} & \cC^\ell_i \arrow{d} \\
      \PP^1 \arrow{r} \arrow{d} & \Spec k_{\NN^2} \arrow{d} \\
      Bl_0(\AA^2) \arrow{r}{\phi} & \AA^2
    \end{tikzcd}
  \end{figure}

  Where the bottom row has a canonical perfect obstruction theory. This induces a relative dimension zero Gysin map for the top row, so we may compare $[\cC^\ell_j]$ and $\phi^! ([\cC^\ell_i])$ inside $\phi^* \cC^\ell_i$. In this case we are simply pulling back and capping with the top Chern class of the excess obstruction bundle, which gives $[\cO(-1)]$ embedded inside the trivial rank two bundle, exactly equal to the pushforward of the cone $\cC^\ell_j$.
  
  \end{eg}

The solution therefore seems to be that this cone defines a log class on itself different from the fundamental class. Indeed for a large class of targets it does:

\begin{thm}
\label{thm:normalcone}
  Let $X \to Y$ be a log morphism such that $Y$ has compatible fundamental class and $X$ locally free and suppose that $\pi: \tilde{X} \to X$ is a locally free log blow-up. Then $\cC^\ell _{\tilde{X}/Y}$ is a subscheme of $\cC^\ell _{X/Y} \times _X \tilde{X}$ with inclusion $i$ and $i_* ([\cC^\ell _{\tilde{X}/Y}]) = \pi^! ([\cC^\ell _{X/Y}])$.
  
\end{thm}

\begin{proof}

  We apply Vistoli's rational equivalence from~\cite{Vistoli}. Since $X$ is locally free there is a strict embedding of $X$ into a log scheme $U$ log smooth over $Y$ which is therefore free in a neighbourhood of the image of $X$, say $U$. The scheme $U$ has compatible fundamental class and the blow-up of $X$ lifts to a blow-up of $U$, $\pi_U: \tilde{U} \to U$. Then $\pi_U$ is an lci morphism, so the normal cone to it is a vector bundle stack, $\cN^\ell_{\tilde{U}/U}$.

  We take the following Cartesian diagram:
  \begin{figure}[h]\centering\begin{tikzcd}
    \tilde{X} \arrow{r} \arrow{d} & \tilde{U} \arrow {d} \\ X \arrow{r} & U
  \end{tikzcd}\end{figure}
  
  \noindent In~\cite{Vistoli} Vistoli produces an explicit rational equivalence:
  \[ [\cC^\ell_{\cC^\ell_{X/U} \times_U \tilde{U} / \cC^\ell_{X/U}}] = [\cC^\ell_{\cN^\ell_{\tilde{U}/U} \times_U X / \cN^\ell_{\tilde{U}/U}}] \]
  inside $A_* (\cC^\ell_{X/U} \times_U \cN^\ell_{\tilde{U}/U})$. Pulling back this relation along the zero section of $\cN^\ell_{\tilde{U}/U}$ gives an equality $[\pi^! \cC^\ell_{X/U}] \cong i_*[\cC^\ell_{\tilde{X}/\tilde{U}}]$ inside $A_* (\tilde{X} \times_X \cC^\ell_{X/U})$.

  The log cotangent bundle of $U$ over $Y$ pulls back to that on $\tilde{U}$ since the map $U \to Y$ is log smooth and $\tilde{U} \to U$ log \'etale. Therefore the above equivalence descends to one in the quotient cone.
    
\end{proof}

As a corollary one obtains:

\begin{cor}
\label{cor:conedefined}
  Let $f: X \to Y$ be a log morphism with $X$ locally free and $Y$ having compatible fundamental class. Let $\cC^\ell_{X/Y}$ be the log normal cone with log structure pulled back from $X$. Then there is a canonical class in $\cC^\ell_{X/Y}$ given by the fundamental class. 
  
\end{cor}

Another corollary answers a question of Herr in~\cite{LeosPaper} on the existence of Gysin maps.

\begin{con}

We extend the construction of~\cite{LeosPaper} Section 3 to produce Gysin maps in the classical sense defined on the level of cycle classes. Let $f: X \to Y$ be a log morphism. The data of a rank $d$ log perfect obstruction theory is the same as the data of a rank $d$ perfect obstruction theory for $X \to \cA_{X/Y}$. The map $X \to \cA_{X/Y}$ is strict, and so the construction of~\cite{VirtualPullbacks} produces a Gysin map in $A_\dagger ^d (X \to \cA_{X/Y})$ which we call the \emph{Herr-Gysin pullback}. Composing with the log \'etale map $\cA_{X/Y} \to Y $ we obtain a Gysin map in $A_\dagger ^d (X \to Y)$.

When $Y$ is equidimensional and has maximal pure log dimension the stack $\cA_{X/Y}$ is also equidimensional of the same dimension, it admits an \'etale cover by log blow-ups of $Y$. By the above corollary the Gysin map is stable under further log blow-ups of $Y$ or $X$. This pure dimensionality is automatically the case when $Y$ is connected and log smooth over the trivial log point. 

\end{con}

When $X \to Y$ is strict the log normal cone is given by the normal cone, and this strictness often holds for many log moduli problems. We can give a refinement of the phenomena seen in~\cite{LogarithmicGromovWittenTheoryWithExpansions} using the following corollary of~\ref{thm:normalcone}.

\begin{cor}
\label{cor:DhruvRelation}
  Let $\cM$ be a locally free finite-type log stack with a map $\cM \to \frM$ to a log smooth locally finite-type moduli stack $\frM$ and a perfect obstruction theory $\EE$. Let $\pi: \tilde{\frM} \to \frM$ be a locally free log \'etale morphism and $\tilde{\cM}$ the fibre product. Then there is an equality $\pi^![\cM]^{vir} = [\tilde{\cM}]^{vir}$.   
  
\end{cor}

\begin{eg}

  Suppose that $\frM$ is the moduli stack of curves $\frM^{log}_{g,n}$ and $\cM_{g,n} \to \frM_{g,n}^{log}$ is the strict classifying map for a space of stable maps. Then we may choose a log blow-up of an finite-type open substack $\pi: \widetilde{\frM}_{g,n}^{log} \to \frM_{g,n}^{log}$ such that $\widetilde{\frM}_{g,n}^{log}$ is locally free. We write $\widetilde{\cM}_{g,n}$ for the fibre product which is locally free and we may arrange to have compatible fundamental class. This stack inherits a virtual fundamental class. The pushforward of the fundamental class of $\widetilde{\frM}_{g,n}^{log}$ is the fundamental class of $\frM_{g,n}^{log}$ and so $\pi_* ([\widetilde{\cM}_{g,n}]^{vir}) = [\cM_{g,n}]^{vir}$.

  Now suppose that $\frN_{g,n}^{log} \to \frM_{g,n}^{log}$ is another log blow-up such that the pullback of $\cM_{g,n}$, which we write $\cN_{g,n}$. Choose a further blow-up of an open finite-type substack $\psi: \widetilde{\frN}_{g,n}^{log} \to \frN_{g,n}^{log}$ with $\widetilde{\frN}^{log}_{g,n}$ locally free and admitting a map to $\widetilde{\frM}_{g,n}^{log}$. Then we have a commuting cube:

  \[
    \begin{tikzcd}[row sep=1.5em, column sep = 1.5em]
    \widetilde{\cN}_{g,n} \arrow[rr] \arrow{dr}{\phi} \arrow[dd] &&
    \cN_{g,n} \arrow[dd] \arrow[dr] \\
    & \widetilde{\cM}_{g,n} \arrow[dd] \arrow[rr] &&
    \cM_{g,n} \arrow[dd] \\
    \widetilde{\frN}_{g,n}^{log} \arrow[rr,] \arrow[dr] && \frN_{g,n}^{log} \arrow[dr] \\
    & \widetilde{\frM}_{g,n}^{log} \arrow[rr] && \frM_{g,n}^{log}
    \end{tikzcd}
    \]

    \noindent satisfying $\pi_* ([\widetilde{\cN}_{g,n}]^{vir}) = [\cN_{g,n}]^{vir}$ and by the above $[\widetilde{\cN}_{g,n}]^{vir} = \phi^! [\widetilde{\cM}_{g,n}]^{vir}$. This implies $\phi_* [\widetilde{\cN}_{g,n}]^{vir} =  [\widetilde{\cM}_{g,n}]^{vir}$. Pushing forward this relation to $A_* (\cM_{g,n})$ we obtain the compatibility of~\cite{LogarithmicGromovWittenTheoryWithExpansions} Proposition 3.6.2.
    
\end{eg}

\bibliographystyle{alpha}
\bibliography{Uploaded}

\end{document}